\documentclass[a4paper,
reqno]{amsart}
\usepackage{amsmath}
\usepackage{cases}
\usepackage{amsfonts}
\usepackage[colorlinks,linkcolor=blue,citecolor=blue]{hyperref}
\usepackage{latexsym, amssymb, amsmath, amsthm, bbm}
\usepackage[all]{xy}
\usepackage{pgfplots}
\usepackage{enumerate}
\usepackage{tikz-cd}

\DeclareSymbolFont{EulerExtension}{U}{euex}{m}{n}
\DeclareMathSymbol{\euintop}{\mathop} {EulerExtension}{"52}
\DeclareMathSymbol{\euointop}{\mathop} {EulerExtension}{"48}

\allowdisplaybreaks[4]

\def \id{\operatorname{id}}

\def \C{\mathcal{C}}

\def \k{\mathbbm{k}}

\def \dim{\operatorname{dim}}
\def \FPdim{\operatorname{FPdim}}
\def \End{\operatorname{End}}
\def \rad{\operatorname{rad}}

\def \Hom{\operatorname{Hom}}

\def \Im{\operatorname{Im}}

\def \C{\mathcal{C}}

\def \1{\mathbbm{1}}

\def \End{\operatorname{End}}

\def \mod{\mathsf{mod}}
\def \-{\text{-}}
\def \Coker{\operatorname{Coker}}
\numberwithin{equation}{section}

\newtheorem{theorem}{Theorem}[section]
\newtheorem{lemma}[theorem]{Lemma}
\newtheorem{proposition}[theorem]{Proposition}
\newtheorem{corollary}[theorem]{Corollary}
\newtheorem{definition}[theorem]{Definition}

\newtheorem{remark}[theorem]{Remark}

\newtheorem{notation}[theorem]{Notation}

\begin{document}
\title[On the stable equivalences between finite tensor categories]{On the stable equivalences between finite tensor categories}
\thanks{The first author was supported by Postgraduate Research
and Practice Innovation Program of Jiangsu Province Grant KYCX22\_0081. The second author was supported by National
Natural Science Foundation of China (NSFC) grant 12171230}

\author[Y. Xu and G. Liu]{Yuying Xu and  Gongxiang Liu}
\address{Department of Mathematics, Nanjing University, Nanjing 210093, China}
\email{yuyingxu@smail.nju.edu.cn}
\email{gxliu@nju.edu.cn}

\date{}

\begin{abstract}
We aim to study Morita theory for tensor triangulated categories. For two finite tensor categories having no projective simple objects, we prove that their stable equivalence induced by an exact $\k\-$linear monoidal functor can be lifted to a tensor equivalence under some certain conditions.
\end{abstract}

\maketitle

\section{Introduction}
Our aim is to study tensor triangulated categories from a Hopf algebraic perspective. In recent years, there has been tremendous interest in developing tensor triangulated categories (see \cite{Bal05, Bal10, NVY19, NVY22} and references therein).  However, so far, limited work has been done in Hopf algebraic fields.

Since all finite tensor categories are Frobenius categories \cite{EO04}, it follows that their stable categories are in fact tensor triangulated categories. In particular, take two finite dimensional Hopf algebras $H$ and $H'$ and consider their representation categories $H\-\mod$ and $H'\-\mod$. A natural question is: If their stable categories are equivalent as tensor triangulated categories, then what can we say about the relations between $H$ and $H'$?

An important relation in Hopf algebras is gauge equivalence. Ng and Schauenburg showed in \cite{NS08} that $H$ and $H'$ are gauge equivalent if and only if $H\-\mod$ and $H'\-\mod$ are $\k\-$linear monoidally equivalent. Nevertheless, it should be pointed out that even if the stable categories of two finite-dimensional algebras are equivalent, the corresponding algebraic structures may be quite different. That is, Morita theory does not work. For example, it is easy to check that some block algebras are not Morita equivalent, although they are stably equivalent \cite{Bro94}.

Our main results states that, under some mild conditions, for Hopf algebras if their stable equivalence is induced from an exact $\k\-$linear monoidal functor then they are gauge equivalent. We describe this observation in the categorical language as follows:
\begin{proposition}\label{result1}
Let $\C$ and $\C'$ be two non-semisimple finite tensor categories. Suppose $F:\C\rightarrow \C'$ is an exact $\k\-$linear monoidal functor inducing a stable equivalence $\underline{F}:\underline{\C}\rightarrow \underline{\C'}.$ If all simple objects in $\C$ and $\C'$ are invertible, then $F$ is a tensor equivalence.
\end{proposition}
\begin{theorem}\label{result2}
Let $\C$ and $\C'$ be two non-semisimple finite tensor categories having no projective simple objects such that $\FPdim(\C)=\FPdim(\C')$. Suppose $F:\C\rightarrow \C'$ is an exact $\k\-$linear monoidal functor inducing a stable equivalence $\underline{F}:\underline{\C}\rightarrow \underline{\C'},$ then $F$ is a tensor equivalence.
\end{theorem}

The present paper is built up as follows. Some definitions, notations and results related to stable categories, tensor categories and Hopf algebras are presented in Section \ref{section2}. We devote Section \ref{section3} to give a proof to our main results: Proposition \ref{result1} and Theorem \ref{result2}.

\section{Preliminaries}\label{section2}
Let $\k$ be an algebraically closed field throughout this paper. For any $\k$-algebra $A$, the category of finitely generated modules over $A$ is denoted by $A\-\mod$. About general background knowledge,
the reader is referred to \cite{ARS95} for stable categories, \cite{Mon93} for Hopf algebras and \cite{EGNO15} for tensor categories.

\subsection{Stable categories}
Let $\C$ be a $\k$-linear abelian category. The \textit{stable category} of $\C$ written as $\underline{\C}$ is defined as follows: The objects of $\underline{\C}$ are the same as those of $\C$; For any objects $X,Y\in\underline{\C}$, the morphisms from $X$ to $Y$ are given by the quotient space $$\underline{\Hom}_\C(X,Y)=\Hom_\C(X,Y)/\mathcal{P}(X,Y),$$where $\mathcal{P}(X,Y)$ is the subspace of $\Hom_\C(X,Y)$ consisting of homomorphisms which factor through a projective object.
We say two $\k$-linear abelian categories $\C$ and $\C'$ are \textit{stably equivalent}, if $\underline{\C}$ and $\underline{\C'}$ are $\k$-linear equivalent.

For simplicity of presentations, we stipulate the following notations.
\begin{notation}

We use $A\-\underline{\mod}$ to denote the stable category of $A\-\mod$. Talking about any stable categories $\underline{\C}$, the following notations are always used:
\begin{itemize}
    \item For $X, Y\in\underline{\C}$, let $\underline{f}$ denote the morphism in the quotient space $\underline{\Hom}_\C(X,Y)$ represented by $f\in \Hom_\C(X,Y)$. We use the diagram below to indicate $\underline{f}=0:$
    $$f: X\stackrel{i}\rightarrow P\stackrel{j}\rightarrow Y,$$
    where $f=j\circ i$ in $\Hom_{\C}(X,Y)$ and $P$ is a projective object in $\C$.
    \item Given a $\k$-linear functor $F:\C\rightarrow\C'$, if $F$ transforms projective objects to projective objects, then it induces a functor from $\underline{\C}$ to $\underline{\C'}$: $$\underline{F}:\underline{\C}\rightarrow\underline{\C'},\;\;X\mapsto F(X),\;\;\underline{f}\mapsto \underline{F(f)},$$
    where $X\in\underline{\C}$ and $f$ is a morphism in $\underline{\C}$.
\end{itemize}
\end{notation}
Recall that an artin algebra $A$ is said to be \textit{self-injective} if it is injective as an $A$-module. A great deal of mathematical effort in the representation theory of algebras has been devoted to the study of self-injective algebras. The following proposition tells us when a stable equivalence can be lifted to a Morita equivalence. 
\begin{lemma}
\emph{(}\cite[Proposition 2.5]{Lin96}\emph{)}\label{lem:quiv}
Let $A$ and $A'$ be self-injective $\k$-algebras having no projective simple modules and $F:A\text{-}\mod\rightarrow A'\text{-}\mod$ be an exact functor. Suppose $F$ induces a stable equivalence $\underline{F}:A\text{-}\underline{\mod}\rightarrow A'\text{-}\underline{\mod}$. Then $F$ is an equivalence if and only if $F$ maps any simple $A$-module to a simple $A'$-module.
\end{lemma}

\subsection{Tensor categories}

We have the following basic properties about tensor categories.
\begin{lemma}\emph{(}\cite[Propositon 2.3]{EO04}\emph{)}\label{1}
Any projective object in a tensor category is also injective.
\end{lemma}
\begin{lemma}\emph{(}\cite[Corollary 2, p.441]{KL94}\emph{)}\label{2}
Let $P$ be a projective object in a tensor category $\C$. Then $P\otimes X$ and $X\otimes P$ are both projective for any object $X\in\C$.
\end{lemma}

It is known that a \textit{tensor equivalence} is a $\k$-linear monoidal equivalence. Here we state the relations between tensor equivalences and \textit{gauge equivalences} in the case of Hopf algebras:
\begin{lemma}\label{tensor}\emph{(}\cite[Theorem 2.2]{NS08}\emph{)}
Let $H$ and $H'$ be finite-dimensional Hopf algebras over $\k$. If $H\-\mod$ and $H'\-\mod$ are tensor equivalent, then $H$ is gauge equivalent to $H'$ as Hopf algebras.
\end{lemma}
An important technical tool in the study of tensor categories is Frobenius-Perron dimensions. Due to \cite[Proposition 4.5.4]{EGNO15}, one can define the Frobenius-Perron dimensions of objects in a tensor category $\C$. To be specific, for each object $X\in\C$, $\FPdim(X)$ is the largest positive eigenvalue of the matrix of left or right multiplication by $X$. Furthermore, $\FPdim$ is the unique additive and multiplicative map which takes positive values on all simple objects of $\C$. Here is a lemma which we will need later.
\begin{lemma}\emph{(}\cite[Proposition 4.5.7]{EGNO15}\emph{)}\label{FPdim}
Let $\C$ and $\C'$ be finite tensor categories. If $F:\C\rightarrow\C'$ is an exact $\k\-$linear monoidal functor, then $\FPdim(F(X))=\FPdim(X)$ for any $X\in\C$.
\end{lemma}
Let $\{L_i\}_{i\in I}$ be the set of isomorphic classes of simple objects of $\C$, and $P_i$ denotes the projective cover of $L_i$ for each $i$.
\begin{definition}\emph{(}\cite[Definition 6.1.6]{EGNO15}\emph{)}
Let $\C$ be a  finite tensor category. Then the Frobenius-Perron dimension of $\C$ is defined by 
$$\FPdim(\C):=\sum_{i\in I}\FPdim(L_i)\FPdim(P_i)$$

\end{definition}
For a finite dimensional Hopf algebra $H$, it is easy to see  $\FPdim(H\-\mod)=\dim_{\k}(H)$, which can be found in \cite[Example 6.1.9]{EGNO15}.

\subsection{Tensor triangulated categories}
 
In retrospect,  all  finite  tensor  categories  are  Frobenius  categories by Lemma \ref{1}. Meanwhile the stable categories of Frobenius categories are triangulated categories \cite[Theorem 2.6]{Hap88}. 

According to \cite{NVY19}, a \textit{tensor \emph{(}monoidal\emph{)} triangulated category} is a triangulated category having a monoidal structure \cite[Chapter \uppercase\expandafter{\romannumeral7}]{Mac98} 
$$\otimes:\C\times\C\rightarrow\C$$
and a unit object $\mathbf{1}$ $\in\C$, such that the bifunctor $-\otimes-$ is exact in each variable.
Then the stable categories of finite tensor categories are naturally tensor triangulated categories.

Two tensor triangulated categories $\underline{\C}$ and $\underline{\C'}$ are said to be \textit{tensor triangulated equivalent} if there is a monoidal functor making $\underline{\C}$ and  $\underline{\C'}$ be triangulated equivalent.
Our aim in this paper is to show that a tensor equivalence can be recovered by a stable equivalence as a special form of tensor triangulated equivalences. Note that a stable equivalence induced by an exact $\k\-$linear monoidal functor is clearly a tensor triangulated equivalence.

\section{Main results}\label{section3}
In the begining, we turn to mention the relation between the Chevalley property and the existence of simple projective objects. A Hopf algebra is said to have \textit{the Chevalley property}, if the tensor product of two simple modules is semisimple. Generally, let us say a tensor category has \textit{the Chevalley property} if the category of semisimple objects is a tensor subcategory \cite[Definition 4.1]{AEG01}.

The following lemma is contributed to simplify the assumptions of our results.

\begin{lemma}\label{pro}
Let $\C$ be a non-semisimple finite tensor category with the Chevalley property. Then $\C$ has no simple projective objects.
\end{lemma}
\begin{proof}
Otherwise, let $L$ be a simple projective object in $\C$. Since $L\otimes L^*$ is semisimple, $\mathbf{1}$ is a direct summand of it. Moreover, Lemma \ref{2} tells us that $L\otimes L^*$ is projective as $L$ is projective. This implies $\mathbf{1}$ is also projective, then $\C$ is semisimple by \cite[Corollary 4.2.13]{EGNO15}, a contradiction.
\end{proof}

A direct consequence of this lemma is:
\begin{corollary}\label{pro2}
Let $H$ be a finite-dimensional non-semisimple Hopf algebra with the Chevalley property. Then $H\-\mod$ has no simple projective modules.
\end{corollary}

It is easy to see that a tensor category in which every simple object is invertible (in the sense of \cite[Definition 2.11.1]{EGNO15})
has the Chevalley property by \cite[Proposition 4.12.4]{EGNO15}. With this observation, we are in a position to show our first main conclusion now:

\textbf{Proof of Proposition \ref{result1}.}
We claim $F$ maps simple objects to simple objects. Actually, for any simple object $L\in\C$, we have:
$$F(L^*)\otimes F(L)\cong F(L^*\otimes L)\cong F(\k)\cong\k$$
then $$\mathrm{length}(F(L^*))\mathrm{length}(F(L))\leq \mathrm{length}(F(L^*)\otimes F(L))=\mathrm{length}(\k)=1,$$
where $\mathrm{length}(\-)$ denotes the length of the Jordan-H\"{o}lder series.
Hence $\mathrm{length}(F(L))=1$, that is, $F(L)$ is a simple object. 

Since $\C$ and $\C'$ are finite, we may assume $\C\cong A\-\mod$, $\C'\cong A'\-\mod$ as $\k$-linear abelian categories, where $A$ and $A'$ are finite-dimensional $\k$-algebras. In addition, as $\C$ and $\C'$ are tensor categories, $A$ and $A'$ also can be self-injective according to Lemma \ref{1}. Moreover, $\C$ and $\C'$ have no projective simple objects by Lemma \ref{pro}. As a result, $F$ is a $\k$-linear equivalence by Lemma \ref{lem:quiv}. Consequently it is a tensor equivalence. \qed

Note that a Hopf algebra $H$ is \textit{basic} if and only if every simple object in the tensor category of finite-dimensional $H\-$modules is invertible. So the following conclusion is directly obtained.
\begin{corollary}\label{basic}
Let $H$ and $H'$ be  finite-dimensional non-semisimple basic Hopf algebras. Suppose $F:H\-\mod\rightarrow H'\-\mod$ is an exact $\k\-$linear monoidal functor inducing a stable equivalence $\underline{F}:H\-\underline{\mod}\rightarrow H'\-\underline{\mod}$. Then $H$ and $H'$ are gauge equivalent. 
\end{corollary}

To prove our second main results, we need several lemmas.
First, let us make some basic observations.

\begin{lemma}\label{projective}
Let $\C$ be a non-semisimple  finite Frobenius $\k\-$linear abelian category. 
\begin{itemize}
    \item [(1)] Let $f:X\rightarrow Y$ be an epimorphism in $\C$. If $\underline{f}=0$ in $\underline{\C}$, then 
$f$ has the following form: 
$$f: X\stackrel{i}\rightarrow P(Y)\stackrel{p}\twoheadrightarrow Y,$$
where $(P(Y),p)$ is a projective cover of $Y$ and $f=p\circ i$.
  \item [(2)] Let $g:X\rightarrow Y$ be a monomorphism in $\C$. If $\underline{g}=0$ in $\underline{\C}$, then 
$f$ has the following form: 
$$g: X\stackrel{i'}\rightarrowtail I(X)\stackrel{p'}\rightarrow Y,$$
where $(I(X),i')$ is an injective hull of $X$ and $g=p'\circ i'$.
\end{itemize}
\end{lemma} 

\begin{proof}
\begin{itemize}
    \item [(1)]
According to $\underline{f}=0$ in $\underline{\C}$, we can find a projective object $P$ such that $f=\beta\circ\alpha$, where
$$f: X\stackrel{\alpha}\rightarrow P\stackrel{\beta}\twoheadrightarrow Y.$$
Moreover, since $f$ is an epimorphism, so is $\beta$.
By the universal property of projective cover, there exists an epimorphism $h:P\twoheadrightarrow P(Y)$ such that $p\circ h=\beta$. 

As a result, we have:
$$f: X\stackrel{h \alpha}\rightarrow P(Y)\stackrel{p}\twoheadrightarrow Y.$$ 
\item [(2)] We omit the proof, which is similar to (1).

\end{itemize}
\end{proof}
Next result is a categorical version of a result in representation theory of artin algebras. 
\begin{lemma}\emph{(}\cite[cf. Proposition 1.1, p.336]{ARS95}\emph{)}\label{ind}
Let $\C$ and $\C'$ be two non-semisimple finite $\k\-$linear abelian categories and $F:\C\rightarrow \C'$ be a $\k\-$linear functor inducing a stable equivalence $\underline{F}:\underline{\C}\rightarrow \underline{\C'}$. Then $F$ gives a one to one correspondence between the isoclasses of indecomposable non-projective objects in $\C$ and $\C'$.
\end{lemma}
\begin{proof}
For the reason that $\C$ and $\C'$ are finite $\k\-$linear abelian categories, we can assume $\C\cong A\-\mod$, $\C'\cong A'\-\mod$ as $\k$-linear abelian categories, where $A$ and $A'$ are finite-dimensional $\k$-algebras. For any $A\-$module $X$, we have the result that $\mathcal{P}(X,X)\subseteq \rad \End_A(X)$ if and only if $X$ has no non-zero projective direct summand (See \cite[Proposition 2.5]{Aus69}.). It follows that $\End_A(X)$ is local if and only if $\End_{A'}(X')$ is local where $F(X)\cong X'\oplus P'$ satisfying that $X'$ has no non-zero projective direct summand and $P'$ is projective. That is, $X$ is indecomposable if only if $X'$ is indecomposable.

Hence we have the following one to one correspondence
$$
\xymatrix{
{\left\{
\begin{array}{c}
\text{Isoclasses of indecomposable} \\
\text{non-projective $A$-modules}
\end{array}
\right\} }
\ar@<0.7ex>[r]^{\Phi}
&
{\left\{
\begin{array}{c}
\text{Isoclasses of indecomposable} \\
\text{non-projective $A'$-modules}
\end{array}
\right\} }
\ar@<0.7ex>[l]^{\Psi} 
}.
$$
Specifically, for any indecomposable non-projective $A$-module $X$, we define $\Phi(X)=X'$ satisfying that $F(X)\cong X'\oplus P'$ for some projective $A'$-module $P'$. Conversely, for any indecomposable non-projective $A'$-module $Y'$, we define $\Psi(Y')=Y$ satisfying that $F(Y)\cong Y'\oplus Q'$ for some projective $A'$-module $Q'$. It is directly to see $\Phi$ and $\Psi$ are well-defined by Krull-Schmidt Theorem. Moreover, $\Phi\circ\Psi=\id$ and $\Psi\circ\Phi=\id$.
The proof is completed.
\end{proof}
Under the assumption of Lemma $\ref{ind}$, there is a pair of mutually inverse maps still denoted by $\Phi$ and $\Psi$
$$
\xymatrix{
{\left\{
\begin{array}{c}
\text{Isoclasses of indecomposable} \\
\text{non-projective objects in $\C$}
\end{array}
\right\} }
\ar@<0.7ex>[r]^{\Phi}
&
{\left\{
\begin{array}{c}
\text{Isoclasses of indecomposable} \\
\text{non-projective objects in $\C'$}
\end{array}
\right\} }
\ar@<0.7ex>[l]^{\Psi} 
}.
$$

Using the above lemma, we have the following result.
\begin{lemma}\label{simple}
Let $\C$ and $\C'$ be non-semisimple  finite Frobenius $\k\-$linear abelian categories. Suppose a $\k\-$linear functor $F:\C\rightarrow\C'$ induces a stable equivalence between $\C$ and $\C'$. 
\begin{itemize}
    \item [(1)] For any indecomposable non-projective object $X\in\C$ and any simple object $L'\in\C'$, we have $L'$ is a quotient object of $\Phi(X)$ if and only if $\underline{\Hom}_{\C'}(\Phi(X),L')\neq0.$
    \item [(2)] For any indecomposable non-projective object $Y'\in\C'$ and any simple object $L\in\C$, we have $L$ is a subobject of $\Psi(Y')$ if and only if $\underline{\Hom}_{\C}(L,\Psi(Y'))\neq0.$
\end{itemize}
\end{lemma} 
\begin{proof}
\begin{itemize}
    \item [(1)]
``Only if\," part: We claim the epimorphism $f:\Phi(X)\rightarrow L'$ satisfies $\underline{f}\neq 0$, which would follow that $\underline{\Hom}_{\C'}(\Phi(X),L')\neq0$. First, we note that $L'$ must be non-projective as $\Phi(X)$ is indecomposable and non-projective. Assume on the contrary that $f$ has the following form: $$f:\Phi(X)\stackrel{i}\rightarrow P(L')\stackrel{j}\twoheadrightarrow L'$$ where $P(L')$ can be chosen as a projective cover of $L'$ by Lemma \ref{projective} (1).
Let us consider the following commuting diagram:
$$
\begin{tikzcd}
{\Phi(X)} \arrow[r,"i"] & {P(L')} \arrow[r, two heads,"t"] \arrow[d, two heads,"j"'] & {\Coker(i)} \arrow[d,two heads,"\beta"] \\
                   & {L'} \arrow[r,two heads,"\alpha"]                                 & {N}          
\end{tikzcd}
$$
where $(\Coker(i),t)$ is the cokernel of $i$ and $(\beta,\alpha)$ is the pushout of $(j,t)$.  

There are two cases which may happen:
\begin{itemize}
  \item[(i)] If $N=0$, then we have an epimorphism:
  $$P(L')\twoheadrightarrow \Coker(i) \oplus L',$$
  which follows another epimorphism: $$P(L')=P(P(L'))\twoheadrightarrow P(\Coker(i)) \oplus P(L')$$
  where $P(P(L'))$ and $P(\Coker(i))$ denote projective covers of $P(L')$ and $\Coker(i)$ respectively.
  Thus, $\Coker(i)=0$ and consequently $P(L')$ is a direct summand of $\Phi(X)$, which contradicts to the fact that $\Phi(X)$ is indecomposable and non-projective.
  
  \item[(ii)] If $N\neq 0$,
  since $\alpha\circ f=\alpha\circ j\circ i=\beta\circ t\circ i=0$, we find $\alpha=0$. This leads to a contradiction that $N=\Im(\alpha)$.
\end{itemize}

In conclusion, $\underline{f}\neq 0$ and thus $\underline{\Hom}_{\C'}(\Phi(X),L')\neq0$. 

``If\," part: Conversely, $\underline{\Hom}_{\C'}(\Phi(X),L')\neq0$ makes $\Hom_{\C}(\Phi(X),L')\neq0$ which deduces that $L'$ is a quotient object of $\Phi(X)$. 
\item[(2)] The proof of this result is dual to that given above by using pullback instead and so is omitted.
\end{itemize}
\end{proof}

\begin{corollary}\label{iff}
Let $\C$ and $\C'$ be non-semisimple  finite Frobenius $\k\-$linear abelian categories having no projective simple objects. Suppose a $\k\-$linear functor $F:\C\rightarrow\C'$ induces a stable equivalence between $\C$ and $\C'$. For two simple objects $L\in\C$ and $L'\in\C'$, $L$ is a subobject of $\Psi(L')$ if and only if $L'$ is a quotient object of $\Phi(L)$.
\end{corollary}
\begin{proof}
Since $F$ induces a stable equivalence, we have $$\underline{\Hom}_{\C}(L,\Psi(L'))\cong
\underline{\Hom}_{\C'}(F(L),F(\Psi(L')))\cong\underline{\Hom}_{\C'}(\Phi(L),\Phi(\Psi(L')))\cong\underline{\Hom}_{\C'}(\Phi(L),L').$$
Therefore $\underline{\Hom}_{\C}(L,\Psi(L'))\neq0$ if and only if $\underline{\Hom}_{\C'}(\Phi(L),L')\neq0.$ The conclusion is obtained by Lemma \ref{simple}.
\end{proof}
Let $\C$ and $\C'$ be non-semisimple  finite Frobenius $\k\-$linear abelian categories having no projective simple objects. Let $\{L_i\}_{i\in I}$ and $\{L'_j\}_{j\in J}$ be the isoclasses of simple objects in $\C$ and $\C'$ respectively. We introduce the following notation 
$$J_i=\{j\in J\;|\;L_j'\;\text{is a quotient object of}\;\Phi(L_i)\}\;\;\;\;(i\in I).$$

\begin{corollary}\label{important}
Let $\C$ and $\C'$ be non-semisimple  finite Frobenius $\k\-$linear abelian categories having no projective simple objects. Suppose a $\k\-$linear functor $F:\C\rightarrow\C'$ induces a stable equivalence between $\C$ and $\C'$. Then 
$J=\underset{i\in I}{\bigcup}J_i$.
\end{corollary}
\begin{proof}
It is suffices to prove 
$J\subset\underset{i\in I}{\bigcup}J_i.$
Indeed, let $L_j'$ be a simple object in $\C'$ and suppose that $L_i$ is a simple subobject of $\Psi(L_j')$. Therefore $L_j'$ is a simple quotient object of $\Phi(L_i)$ by Corollary \ref{iff}. In other words, $j\in J_i$ for some $i\in I$.
\end{proof}

With the help of the preceding lemmas, we can now prove the following result.

\textbf{Proof of Theorem \ref{result2}.}
Let $\{L_i\}_{i\in I}$ and  $\{L'_j\}_{j\in J}$ be the set of
isoclasses of simple objects in $\C$ and $\C'$ respectively, where $I$ and $J$ are finite sets. Moreover, we use $P_i$ (resp. $P'_j$) to represent a projective cover of each simple object $L_i$ (resp. $L_j'$).   

The trick of the proof is to show $F$ maps simple objects to simple objects. For any simple object $L_i$, we have $F(L_i)\cong\Phi(L_i)\oplus Q'_i$ for some projective object $Q'_i$ by Lemma \ref{ind}. In addition, as $F$ is an exact functor, there is an epimorphism $F(P_i)\twoheadrightarrow P(\Phi(L_i))$ for any $i\in I$, where $P(\Phi(L_i))$ denotes a projective cover of $\Phi(L_i)$. 

Consequently, we can get the following formula:
$$
\begin{aligned}
\FPdim(\C)&=\sum_{i\in I}\FPdim(L_i)\FPdim(P_i)=\sum_{i\in I}\FPdim(F(L_i))\FPdim(F(P_i))\\
&=\sum_{i\in I}\FPdim(\Phi(L_i)\oplus Q'_i)\FPdim(F(P_i))\\
&\geq\sum_{i\in I}\FPdim(\Phi(L_i))\FPdim(P(\Phi(L_i)))\\
&\geq\sum_{i\in I}(\sum_{j\in J_i}\FPdim(L_j'))(\sum_{j\in J_i}\FPdim(P'_j)))\\
&\geq\sum_{j\in J}\FPdim(L_j')\FPdim(P_j')\;\;\;\; \text{(by Corollary \ref{important})}\\
&=\FPdim(\C').\\
\end{aligned}
$$
By the condition that $\FPdim(\C)=\FPdim(\C')$, all the $``\geq"$ above are in fact equalities.
Due to $$\sum_{i\in I}\FPdim(\Phi(L_i)\oplus Q'_i)\FPdim(F(P_i))=\sum_{i\in I}\FPdim(\Phi(L_i))\FPdim(P(\Phi(L_i))),$$
we can deduce that $Q_i'=0$ for any $i\in I.$ Moreover, by $$\sum_{i\in I}(\sum_{j\in J_i}\FPdim(L'_j))(\sum_{j\in J_i}\FPdim(P'_j))\\
=\sum_{j\in J}\FPdim(L_j')\FPdim(P_j'),$$
it is clear that each $J_i$ has just one element for $i\in I.$ 
Without loss of generality, let $J_i=\{L'_{\varphi(i)}\}$ where $\varphi:I\rightarrow J$ is a surjection given by Corollary \ref{important}. At last,  
$$
\begin{aligned}
\sum_{i\in I}\FPdim(L'_{\varphi(i)})\FPdim(P'_{\varphi(i)})&=\sum_{j\in J}\FPdim(L_j')\FPdim(P_j')\\
&=\sum_{i\in I}\FPdim(\Phi(L_i))\FPdim(P(\Phi(L_i))),
\end{aligned}
$$ 
it follows that $\FPdim(\Phi(L_i))=\FPdim(L'_{\varphi(i)})$ for $i\in I.$ Hence $F(L_i)\cong \Phi(L_i)\cong L'_{\varphi(i)}$ for any $i\in I$.

Since $\C$ and $\C'$ are finite, using the same method used in proof of Proposition \ref{result1}, we can assume $\C\cong A\-\mod$, $\C'\cong A'\-\mod$ as $\k$-linear abelian categories, where $A$ and $A'$ are finite-dimensional self-injective $\k$-algebras. By Lemma \ref{lem:quiv}, $F$ is a $\k$-linear equivalence. Consequently it is a tensor equivalence.
\qed

It is direct to see the following corollary. 
\begin{corollary}
Let $H$ and $H'$ be finite-dimensional non-semisimple Hopf algebras having no simple projective modules such that $\dim_{\k}(H)=\dim_{\k}(H')$. If an exact $\k\-$linear monoidal functor $F:H\-\mod\rightarrow H'\-\mod$ induces a stable equivalence $\underline{F}:H\-\underline{\mod}\rightarrow H'\-\underline{\mod}$, then $H$ and $H'$ are gauge equivalent.
\end{corollary}
\begin{proof}
By $\FPdim(H\-\mod)=\dim_{\k}(H)$ we can get the conclusion.
\end{proof}

\begin{remark}
\emph{We end this section by pointing out that: The condition ``the functor $F$ is monoidal'' can not be removed in Theorem \ref{result2}. Let us illustrate it with an example.} 
\emph{Consider the $n^2$-dimensional Taft algebras $T_{n^2}(\omega_1)$ and $T_{n^2}(\omega_2)$, where $\omega_1$ and $\omega_2$ are primitive $n$-th roots of unity. \cite[Corollary 3.3]{KMN12} tells us that $T_{n^2}(\omega_1)$ and $T_{n^2}(\omega_2)$ are gauge equivalent if and only if $\omega_1=\omega_2$. As the fact that $T_{n^2}(\omega_1)$ and $T_{n^2}(\omega_2)$ are isomorphic as algebras, they are Morita equivalent inducing a functor from $T_{n^2}(\omega_1)\-\mod$ to $T_{n^2}(\omega_2)\-\mod$. This functor satisfies the assumptions of Theorem \ref{result2}, except that $F$ is a monoidal functor when $\omega_1\neq\omega_2$.}
\end{remark}

\section*{Acknowledgment}
The authors would like to thank the referee for his/her detailed and valuable comments which improve the paper greatly.

\end{document}